\documentclass[12pt]{article}
\usepackage[utf8]{inputenc}
\usepackage{graphics}
\usepackage{amsmath,enumerate}
\usepackage{amsfonts}
\usepackage{amssymb}
\usepackage{amsthm}
\usepackage{layout}
\usepackage{tikz}
\usetikzlibrary{graphs}
\usetikzlibrary{graphs.standard}
\usepackage{relsize}
\usepackage{pgfplots}
\usepackage[shortlabels]{enumitem}

\usepackage{amsthm}

\newtheorem{lemma}{Lemma}

\newtheorem{theorem}{Theorem}
\newtheorem*{conjecture}{Conjecture}
\newtheorem*{question}{Question}

\theoremstyle{definition}

\newtheorem*{cor}{Corollary}

\title{\vspace{-2cm}On the number of $k$-gons in finite projective planes}
\author{Vladislav Taranchuk\thanks{Department of Mathematical Sciences, University of Delaware.} \\
}
\date{\today}

\begin{document}
\addtolength{\jot}{0.7em}
\maketitle
\begin{abstract}
    Let $\Pi$ be a projective plane of order $n$ and $\Gamma_{\Pi}$ be its Levi graph (the point-line incidence graph). For fixed $k \geq 3$, let $c_{2k}(\Gamma_{\Pi})$ denote the number of $2k$-cycles in $\Gamma_{\Pi}$. In this paper we show that 
    $$
    c_{2k}(\Gamma_{\Pi}) = \frac{1}{2k}n^{2k} + O(n^{2k-2}),  \hspace{0.5cm} n \rightarrow \infty.
    $$
    We also state a conjecture regarding the third and fourth largest terms in the asymptotic of the number of $2k$-cycles in $\Gamma_{\Pi}$. This result was also obtained independently by Voropaev \cite{Vor} in 2012.
    
    Let $\text{ex}(v, C_{2k}, \mathcal{C}_{\text{odd}}\cup \{C_4\})$ denote the greatest number of $2k$-cycles amongst all bipartite graphs of order $v$ and girth at least 6. As a corollary of the result above, we obtain
    $$
    \text{ex}(v, C_{2k}, \mathcal{C}_{\text{odd}}\cup \{C_4\}) = \left(\frac{1}{2^{k+1}k}-o(1)\right)v^k, \hspace{0.5cm} v \rightarrow \infty.
    $$
\end{abstract}
\section{Introduction}

Over the years, many questions have surfaced regarding counting the number of certain substructures within a projective plane. In this paper we contribute to an open question in the area. We omit the standard definitions related to finite geometries and graph theory. For all undefined notions in finite geometries we refer the reader to Casse \cite{Casse}, and to Bollobas \cite{BG} for all graph theoretic notions. We will also need the following definitions and notation.

Let $\Pi$ denote a projective plane of order $n$. Then $N = n^2 + n +1$ represents the number of points and the number of lines in $\Pi$. If $A$ and $B$ are points of  $\Pi$, we write $AB$ for the line containing them. We write $S_k$ for the group of all permutations of $\{ 1, 2, \dots, k \}$, the symmetric group.


The \textit{point-line incidence graph $\Gamma_{\Pi}$} of $\Pi$, also known as the \textit{Levi graph} of $\Pi$, is the bipartite graph with the set of points of $\Pi$ to be one vertex part and the set of lines of $\Pi$ to be the other vertex part. A point $P$ is adjacent to a line $\ell$ in $\Gamma_{\Pi}$ if $P$ lies on $\ell$ in $\Pi$.
We write $P \sim \ell$ to denote adjacency of a point and line in $\Gamma_{\Pi}$.


Let $H$ be a graph and $\mathcal{F}$ be a family of (forbidden) graphs. Let ex$(n, H, \mathcal{F})$ denote the maximum number of copies of $H$ in an $n$-vertex graph containing no graphs in $\mathcal{F}$ as a subgraph. When $H = K_2$ (just an edge), then a simplified notation is used for ex$(n, K_2, \mathcal{F})$, namely ex$(n, \mathcal{F})$, and it is  often called the \textit{Tur\'{a}n number} of $\mathcal{F}$. Clearly, ex$(n, \mathcal{F})$ denotes the largest number of edges an $n$-vertex graph can have without containing any subgraphs from $\mathcal{F}$. Any $n$-vertex graph with ex$(n, \mathcal{F})$ edges is called an extremal graph for $\mathcal{F}$. The problem of determining ex$(n, \mathcal{F})$ is usually referred to as a \textit{Tur\'{a}n type} problem. For the extensive literature related to Tur\'{a}n type problems, see Bollobas \cite{Bol2}, F\"{u}redi \cite{Fur}, F\"{u}redi and Simonovits \cite{FuS}, Verstra\"{e}te \cite{Ver}, Mubayi and Verstra\"{e}te \cite{MVer}, Lazebnik, Sun, and Wang \cite{FelSun}.

Let $C_k$ denote a cycle of length $k$, $\mathcal{C}_k = \{ C_3, C_4, \dots, C_k \}$, and $\mathcal{C}_{\text{odd}} = \{ C_3, C_5, C_7, \dots \}$ the set of all odd cycles. Some early attention that ex$(n, H, \mathcal{F})$ received was from Erd\"{o}s \cite{Erd1} who stated a conjecture regarding the extremal graph of ex$(n, C_5, \{ C_3 \})$. This conjecture was resolved by Hatami, Hladk\'{y}i, Kr\'{a}l, Norine, and Razborov \cite{Hatam} and independently by Grzesik \cite{Grez}, building on the work of Gy\"{o}ri \cite{Gyor}. 
The more recent wave of interest in ex$(n, H, \mathcal{F})$ was initiated by Alon and Shikelman \cite{Aloshil}. 
There have been several new results regarding the growth rate of ex$(n, H, \mathcal{F})$ where $\mathcal{F} = \mathcal{C}_{2m}$ or $\mathcal{F} = \{ C_{2m} \}$, with resolution up to the leading term in certain cases. We refer the reader to the papers \cite{GerPal},  \cite{GerG}, and \cite{SolW} for the most up to date reading regarding ex$(n, H, \mathcal{C}_{2m})$ and ex$(n, H, \{ C_{2m} \})$.


Note that $(\mathcal{C}_{\text{odd}}\cup \{ C_4 \})$-free graphs are bipartite graphs of girth at least 6.
Given that the Levi graph of a projective plane is bipartite and has girth 6, it serves as a good candidate for obtaining lower bounds on ex($n, H, \mathcal{F}$). In this paper, we will be counting the number of cycles of length $2k$ in the Levi graph, and consequently we obtain lower bounds on ex$(v, C_{2k}, \mathcal{C}_{\text{odd}}\cup \{C_4\})$. We denote the number of cycles of length $2k$ in a graph $G$ by $c_{2k}(G)$.

As a simpler problem, one can count the number of closed walks of length $2k$ in the Levi graph of a projective plane $\Pi$ of order $n$. One can easily show that this number depends only on $k$ and $n$, and not on the actual plane. Let $\Pi$ be a projective plane of order $n$ and $\Gamma_{\Pi}$ its Levi graph. Let $A$ be the adjacency matrix of $\Gamma_{\Pi}$. As $A$ is a real symmetric matrix, then all of its eigenvalues are real. Let $\lambda_1 \geq \lambda_2 \dots \geq \lambda_{2N}$ be the eigenvalues of $A$. By considering eigenvalues of $A^2$, it can be deduced that $ \lambda_1 = n+1$ and $\lambda_{2N} = -(n+1)$, each with multiplicity one, and all other eigenvalues are equal to $\pm\sqrt{n}$ each with multiplicity $N - 1$. It follows, see \cite{NBig}, that the number of closed walks of length $2k$ in $\Gamma_{\Pi}$ is given by
$$
\text{Trace}(A^{2k}) = \sum_{i = 1}^{2N} \lambda_i^{2k} = 2(n+1)^{2k}+2(N - 1)n^k
$$
and so the number of closed walks of length $2k$ in $\Gamma_{\Pi}$ depends only on $n$ and $k$, and so it is the same for all projective planes of order $n$.

This may lead one to ask, what other structures appear in a finite projective plane $\Pi$ and does the number of these structures depend only the order of the plane? In general the answer is no, as can be observed in the case of Desargues and Pappus configurations. Another interesting example concerns the number of $k$-arcs. Define a $k$-\textit{arc} in a projective plane $\Pi$ to be a set of $k$ points of $\Pi$, no three of which are collinear. For $k \leq 6$, Glynn \cite{Glyn} showed that the number of $k$-arcs in a plane of order $n$ does not depend on the plane. Furthermore, in \cite{Glyn}, Glynn computes an expression for the number of 7-arcs in any finite projective plane, and using this expression deduces that there do not exist projective planes of order $6$, as evaluating the formula at 6 yields a negative value. Glynn's work counting $k$-arcs was recently extended by Kaplan, Kimport, Lawrence, Peilen and Weinreich \cite{Kap} who determined an expression for the number of 9-arcs in an arbitrary projective plane. It is worth mentioning that for $k = 7, 8, 9$, the formula for the number of $k$-arcs depends on more than just $k$ and the order of the plane.

In \cite{LazMelV1} Lazebnik, Mellinger, and Vega demonstrate that it is possible to embed a $2k$-cycle of every possible size into the Levi graph of any finite affine or projective plane. This was further extended by Aceves, Heywood, Klahr, and Vega \cite{Ace} who showed that one can embed a $2k$-cycle of every possible size into the Levi graph of the projective space $PG(d, q)$. Moreover, in a different paper Lazebnik, Mellinger, and Vega \cite{LazMelV2} motivated the study of counting $2k$-cycles in the Levi graph with the following two questions. For fixed $k \geq 3$:
\begin{enumerate}
    \item Which $C_4$-free bipartite graphs with partitions of size $N$ contain the greatest number of $2k$-cycles?
    \item Do the Levi graphs of all affine or projective planes of order $n$ contain the same number of $2k$-cycles?
\end{enumerate}
Fiorini and Lazebnik \cite{FioL} show that the Levi graph of a projective plane has the largest number of $C_6$'s among all $C_4$-free bipartite graphs with size parts. This work was extended by  
De Winter, Lazebnik, and Verstra\"{e}te \cite{Wint} who show that the same holds when $k = 4$ and $n \geq 157$. In \cite{LazMelV2}, progress towards question 2 is made as the exact value of $c_{2k}(\Gamma_{\Pi})$ is determined for $k = 3, 4, 5, 6$, showing that in these cases, $c_{2k}(\Gamma_{\Pi})$ depends only on the order of $\Pi$. This work was further extended by Voropaev \cite{Vor}, again demonstrating that $c_{2k}(\Gamma_{\Pi})$ depends only on the order of $\Pi$ up to $k = 10$. Determining explicit formulas for larger $k$ may very well have interesting consequences just like in the example of formula for the number of 7-arcs in a projective plane.

In this paper, we make some progress towards resolving question 2 as we determine the first and second leading terms in the asymptotic of the number of $k$-gons in an arbitrary projective plane. The magnitude of the leading term for the number of $k$-gons is shown to be the same as that of the number of closed walks of length $2k$ in the Levi graph of a finite projective plane, but with a different leading coefficient.

Here we list our main results.

\bigskip

\begin{theorem}
Let $\Pi$ be a projective plane of order $n$ and $\Gamma_{\Pi}$ be its Levi graph. Then for fixed $k \geq 4$,
$$
c_{2k}(\Gamma_{\Pi}) = \frac{1}{2k}n^{2k} +O(n^{2k-2}), \hspace{1cm} n \rightarrow \infty
$$
\end{theorem}

\begin{theorem}
Let $k \geq 4$, then
$$
\text{ex}(v, C_{2k}, \mathcal{C}_{\text{odd}} \cup \{ C_4 \}) = \left(\frac{1}{2^{k+1}k}-o(1)\right)v^k, \hspace{1cm} v \rightarrow \infty.
$$
\end{theorem}

The structure of our paper is as follows. In section 2 we state some definitions and prove some important lemmas. In section 3 we begin to place bounds on the number of certain types of subgraphs in $\Gamma_{\Pi}$. In section 4 we prove Theorems 1 and 2 mentioned above. Finally, in our concluding remarks, we give a table of coefficient data on the number of $2k$-cycles in $\Gamma_{\Pi}$ for small $k$ and leave the reader with a conjecture.

\newpage
 
\section{Symmetries of $QG_k$}
Let $\Pi$ be a projective plane and $\Gamma_{\Pi}$ be its Levi graph. For $k \geq 3$, we define a \textit{quasi $k$-gon} to be a sequence $QG_k = (P_1, P_2, \dots, P_k)$ of $k$ distinct points of $\Pi$. Here, $P_k$ and $P_1$ are thought of as consecutive elements in $QG_k$. All arithmetic done in the indices is considered to be modulo $k$, where we will use $\{ 1, 2, \dots k \}$ as the representatives of each equivalence class mod $k$. We call $\mathcal{P}_{QG_k} = \{ P_1, \dots, P_k  \}$ the set of points of $QG_k$ and $\mathcal{L}_{QG_k}$ is the set of all distinct lines amongst $P_1P_2, P_2P_3, \dots, P_kP_1$. For convenience, we will write  $\mathcal{L}_{QG_k} = \{ P_iP_{i+1} : 1 \leq i \leq k \}$. If $|\mathcal{L}_{QG_k}| = k$, meaning all lines of the form $P_iP_{i+1}$ for $1 \leq i \leq k$ are distinct, then we call $QG_k$ a \textit{$k$-gon} and instead denote it by $G_k$. We will denote the number of $k$-gons in a projective plane $\Pi$ by $c_k(\Pi)$. It will be shown later that in fact
$$
\frac{1}{2k}c_{k}(\Pi) = c_{2k}(\Gamma_{\Pi}).
$$


Define the subgraph $\Gamma_{QG_k}$ of $\Gamma_{\Pi}$ corresponding to $QG_k$ as follows: The set of vertices $V(\Gamma_{QG_k})$ is given by $\mathcal{P}_{QG_k} \cup \mathcal{L}_{QG_k}$. The edges $E(\Gamma_{QG_k})$ are obtained by joining a vertex $P_i$ to all vertices in the set $\{ P_{i-1}P_i, P_iP_{i+1} \}$ for $1 \leq i \leq k$. If $P_{i-1}P_i = P_iP_{i+1}$ then $P_i$ has only one neighbor. It is clear that in the case where $QG_k$ is actually a $k$-gon, the corresponding graph $\Gamma_{QG_k}$ is a cycle of length $2k$.

Here we provide an example. Let $QG_7 = (P_1, P_2, \dots, P_7)$ be a quasi 7-gon given by the following figure. We use $QG_7$ to demonstrate the corresponding graph $\Gamma_{QG_7}$.
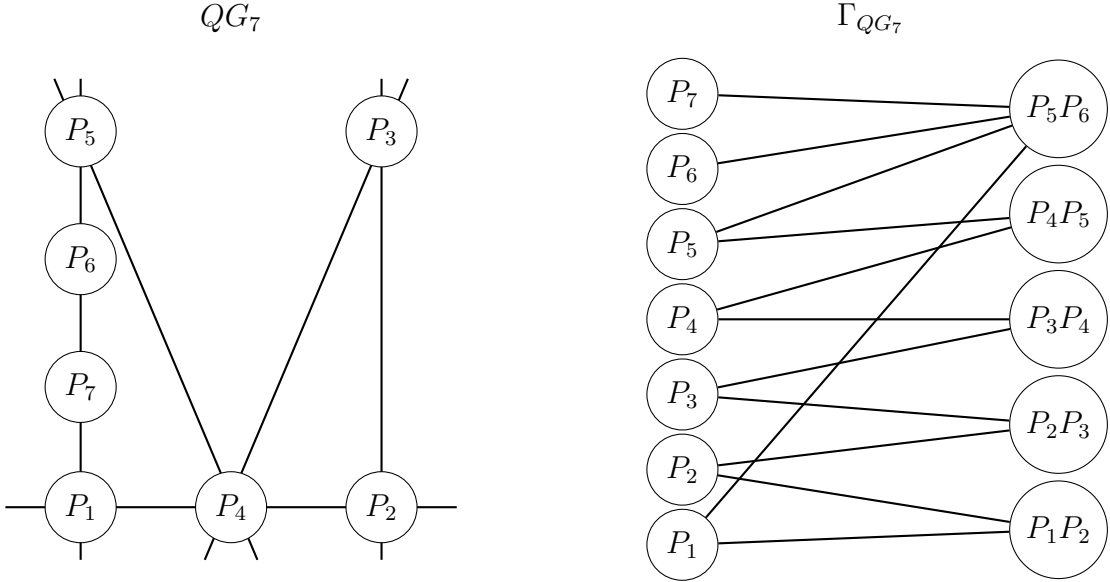
\begin{center}
\begin{tikzpicture}
    \draw [thick, draw = black]
        (0, -0.5) -- (6, -0.5)
        (1, 5.2) -- (1, -1.2)
        (5, 5.2) -- (5, -1.2)
        (2.65, -1.2) -- (5.35, 5.2)
        (3.35, -1.2) -- (0.65, 5.2);

    \node (P5) [circle, fill = white, draw] at (1, 4.5) {$P_5$};
    \node (P3) [circle, fill = white, draw] at (5, 4.5) {$P_3$};
    \node (P1) [circle, fill = white, draw] at (1, -0.5) {$P_1$};
    \node (P4) [circle, fill = white, draw] at (3, -0.5) {$P_4$};
    \node (P2) [circle, fill = white, draw] at (5, -0.5) {$P_2$};
    
    \node (QG) [] at (3, 6) {$QG_7$};
    \node (G) [] at (11.5, 6) {$\Gamma_{QG_7}$};
    
        
    \node (P6) [circle, fill = white, draw] at (1, 2.8) {$P_6$};
    \node (P7) [circle, fill = white, draw] at (1, 1.1) {$P_7$};
    
    \node (P1') [circle, fill = white, draw] at (9, -1) {$P_1$};
    \node (P2') [circle, fill = white, draw] at (9, 0) {$P_2$};
    \node (P3') [circle, fill = white, draw] at (9, 1) {$P_3$};
    \node (P4') [circle, fill = white, draw] at (9, 2) {$P_4$};
    \node (P5') [circle, fill = white, draw] at (9, 3) {$P_5$};
    \node (P6') [circle, fill = white, draw] at (9, 4) {$P_6$};
    \node (P7') [circle, fill = white, draw] at (9, 5) {$P_7$};
    
    \node (L1) [circle, fill = white, draw] at (14, -0.8) {$P_1P_2$};
    \node (L2) [circle, fill = white, draw] at (14, 0.6) {$P_2P_3$};
    \node (L3) [circle, fill = white, draw] at (14, 2.0) {$P_3P_4$};
    \node (L4) [circle, fill = white, draw] at (14, 3.4) {$P_4P_5$};
    \node (L5) [circle, fill = white, draw] at (14, 4.8) {$P_5P_6$};
    
      \draw [ thick, draw=black] 
        (P1') -- (L1)
        (P1') -- (L5);
     \draw [ thick, draw=black] 
        (P2') -- (L1)
        (P2') -- (L2);
    \draw [ thick, draw=black] 
        (P3') -- (L2)
        (P3') -- (L3);
    \draw [ thick, draw=black] 
        (P4') -- (L3)
        (P4') -- (L4);
     \draw [ thick, draw=black] 
        (P5') -- (L4)
        (P5') -- (L5);
     \draw [ thick, draw=black] 
        (P6') -- (L5)
        (P7') -- (L5);
    
\end{tikzpicture}
\end{center}
Let us take a moment to comment on the above figure. Here $QG_7 = (P_1, P_2, \dots, P_7)$, with the corresponding set of lines $\{ P_iP_{i+1} : 1\leq i \leq k \}$. We assume that $P_1P_2, P_2P_3, P_3P_4$, $P_4P_5, P_5P_6$ are all distinct lines. Observe that $P_4$ lies on the line $P_1P_2$, however,  $P_4 \not\sim P_1P_2$ in $\Gamma_{QG_7}$. By definition of $\Gamma_{QG_7}$, we have only that $P_4 \sim P_3P_4$ and $P_4 \sim P_4P_5$. Furthermore, note that $P_5P_6 = P_6P_7 = P_7P_1$ and therefore $P_6$ and $P_7$ each only have one neighbor, namely $P_5P_6$.

The symmetric group $S_k$ acts on quasi $k$-gons in $\Pi$ in the following way: If $QG_k = (P_1, \dots, P_k)$ and $\sigma \in S_k$, then $\sigma(QG_k) := (P_{\sigma(1)}, \dots, P_{\sigma(k)})$. Hence,
$$
\mathcal{P}_{\sigma(QG_k)} = \{ P_{\sigma(1)}, \dots, P_{\sigma(k)}  \} = \{ P_1, \dots, P_k \} = \mathcal{P}_{QG_k}
$$
and the set of lines of $\sigma(QG_k)$ is
$$
\mathcal{L}_{\sigma(QG_k)} =  \{  P_{\sigma(i)}P_{\sigma(i+1)} : 1 \leq i \leq k \}.
$$
Note that in general, $\mathcal{L}_{QG_k}$ is not necessarily equal to $\mathcal{L}_{\sigma(QG_k)}$. 

We call two quasi $k$-gons $QG_k = (P_1, \dots, P_k)$ and $QG_k' = (P_1', \dots, P_k')$ \textit{equivalent}, and write $QG_k \equiv QG_k'$, if $\Gamma_{QG_k} =\Gamma_{QG_k'}$, that is, they have the same vertex set and the same edge set. It is obvious that equivalence of quasi $k$-gons is an equivalence relation and if $QG_k \equiv QG_k'$, then there exists $\sigma \in S_k$ such that $\sigma(QG_k) = QG_k'$. Therefore $S(QG_k) := \{ \sigma \in S_k: QG_k \equiv \sigma(QG_k)  \}$ is a subgroup of $S_k$.

\bigskip

\noindent\textbf{Remark:} Given a quasi $k$-gon $QG_k$ and permutation $\sigma \in S_k$ we stress the following point: We do not view $QG_k$ as a partial plane in $\Pi$ defined by the points and lines of $QG_k$. Therefore, if $QG_k \equiv \sigma(QG_k)$, then $\sigma$ should not be thought of as a collineation. As an example, we refer to the figure above of $QG_7$ and consider $\sigma = (1234567)$. In the lemma that follows, we demonstrate that $QG_7 \equiv \sigma(QG_7)$, however, observe that while $P_5, P_6, P_7, P_1$ lie on one line in $\Pi$, $P_{\sigma(5)} = P_6$, $P_{\sigma(6)} = P_7$, $P_{\sigma(7)} = P_1$ and $P_{\sigma(1)} = P_2$ are not collinear in $\Pi$.

Let $D_k$ denote the dihedral group of order $2k$, which is defined as the group of automorphisms of the graph $C_k$. It is well known that $D_k = \langle a, b\rangle$ where $a^n = b^2 = (ab)^2 = 1$.






\begin{lemma}
Let $\Pi$ be a projective plane and $QG_k = (P_1, \dots, P_k)$ in $\Pi$. Then $S(QG_k)$ contains  $D_k$ as a subgroup.
\end{lemma}

\begin{proof}
Let $\sigma = (12\dots k)$, so that $\sigma(QG_k) = (P_{\sigma(1)}, \dots, P_{\sigma(k)}) = (P_2, \dots, P_k, P_1)$. We wish to show that $\sigma(QG_k) \equiv QG_k$. Note that the vertex set $V(\Gamma_{\sigma(QG_k)}) =  \mathcal{P}_{\sigma(QG_k)} \cup \mathcal{L}_{\sigma(QG_k)} = \mathcal{P}_{QG_k} \cup \mathcal{L}_{\sigma(QG_k)}$. Here
\begin{eqnarray*}
\mathcal{L}_{\sigma(QG_k)}&=& \{  P_{\sigma(i)}P_{\sigma(i+1)} : 1 \leq i \leq k \} =  \{  P_{i+1}P_{i+2} : 1 \leq i \leq k \} \\
&=& \{  P_{i}P_{i+1} : 1 \leq i \leq k \} = \mathcal{L}_{QG_k}.
\end{eqnarray*}
Thus we have $V(\Gamma_{\sigma(QG_k)}) = V(\Gamma_{QG_k})$. The edge set $E(\Gamma_{\sigma(QG_k)})$ is given by joining $P_{\sigma(i)} = P_{i+1}$ to all distinct lines in $\{ P_{\sigma(i-1)}P_{\sigma(i)},  P_{\sigma(i)}P_{\sigma(i+1)} \} = \{ P_{i}P_{i+1}, P_{i+1}P_{i+2} \}$ where $1 \leq i \leq k$. These are exactly the same edges that appear in $\Gamma_{QG_k}$. Thus, $\sigma \in S(QG_k)$ and has order $k$.

Now consider the permutation $\rho$ of $\{ 1, 2, \dots, k \}$, such that $\rho(i) = k+1-i$. That is, $\rho(QG_k) = (P_k, P_{k-1}, \dots, P_1)$. Set $j = k+1 - i$, then $\rho(i) = j$ and $\rho(i+1) = j -1$.
Observe that $V(\Gamma_{\rho(QG_k)}) = \mathcal{P}_{QG_k}\cup \mathcal{L}_{\rho(QG_k)}$ where
\begin{eqnarray*}
\mathcal{L}_{\rho(QG_k)} &=& \{  P_{\rho(i)}P_{\rho(i+1)} : 1 \leq i \leq k \} =  \{  P_{j}P_{j-1} : 1 \leq j \leq k \} \\
&=& \{  P_{j-1}P_{j} : 1 \leq j \leq k \} = \mathcal{L}_{QG_k}.
\end{eqnarray*}
Therefore $V(\Gamma_{\rho(QG_k)}) = V(\Gamma_{QG_k})$. The edge set $E(\Gamma_{\rho(QG_k)})$ is given by joining $P_{\rho(i)} = P_{j}$ to all distinct lines in $\{ P_{\rho(i-1)}P_{\rho(i)},  P_{\rho(i)}P_{\rho(i+1)} \} = \{ P_{j+1}P_{j}, P_{j}P_{j-1} \}$ for $1 \leq j \leq k$. These are exactly the same edges that appear in $\Gamma_{QG_k}$. Thus, $\rho \in S(QG_k)$  and has order $2$.

Now to demonstrate that $D_k$ is in fact the Dihedral group, we show that $(\sigma \rho )^2 = 1$. Since $\sigma(i) = i+1$ and $\rho(i) = k+1 - i$ then $(\sigma\rho)(i) = \sigma(k+1 - i) = k+2 - i$. Therefore 
$$
(\sigma\rho)^2(i) = (\sigma\rho)(k+2 - i) = k+2 - (k+2 - i) = i.
$$
Therefore $\sigma^k = \rho^2 = (\sigma\rho)^2 = 1$, and so $D_k$ is a Dihedral group of order $2k$.
\end{proof}

\noindent We now provide an example showing that it is possible to have $S(QG_k)$ be strictly larger than $D_k$. We refer to our previous example of $QG_7$, and we consider the permutation $\sigma = (67)$.
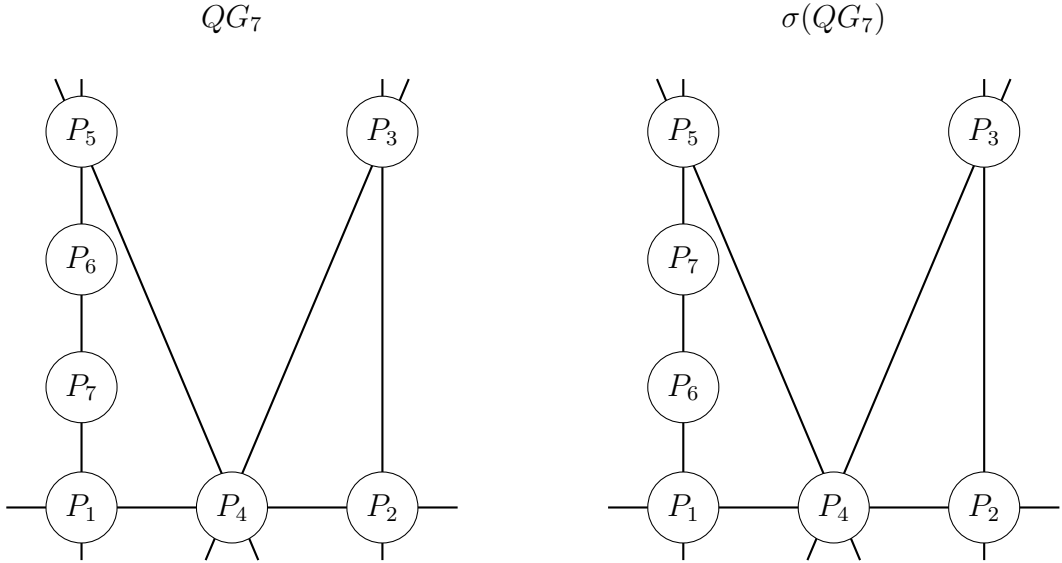
\begin{center}
\begin{tikzpicture}
    \draw [thick, draw = black]
        (0, -0.5) -- (6, -0.5)
        (1, 5.2) -- (1, -1.2)
        (5, 5.2) -- (5, -1.2)
        (2.65, -1.2) -- (5.35, 5.2)
        (3.35, -1.2) -- (0.65, 5.2);

    \node (P5) [circle, fill = white, draw] at (1, 4.5) {$P_5$};
    \node (P3) [circle, fill = white, draw] at (5, 4.5) {$P_3$};
    \node (P1) [circle, fill = white, draw] at (1, -0.5) {$P_1$};
    \node (P4) [circle, fill = white, draw] at (3, -0.5) {$P_4$};
    \node (P2) [circle, fill = white, draw] at (5, -0.5) {$P_2$};
    
    \node (QG) [] at (3, 6) {$QG_7$};
    \node (G) [] at (11, 6) {$\sigma(QG_7)$};
        
    \node (P6) [circle, fill = white, draw] at (1, 2.8) {$P_6$};
    \node (P7) [circle, fill = white, draw] at (1, 1.1) {$P_7$};
    
    \draw [thick, draw = black]
        (8, -0.5) -- (14, -0.5)
        (9, 5.2) -- (9, -1.2)
        (13, 5.2) -- (13, -1.2)
        (10.65, -1.2) -- (13.35, 5.2)
        (11.35, -1.2) -- (8.65, 5.2);

    \node (P5) [circle, fill = white, draw] at (9, 4.5) {$P_5$};
    \node (P3) [circle, fill = white, draw] at (13, 4.5) {$P_3$};
    \node (P1) [circle, fill = white, draw] at (9, -0.5) {$P_1$};
    \node (P4) [circle, fill = white, draw] at (11, -0.5) {$P_4$};
    \node (P2) [circle, fill = white, draw] at (13, -0.5) {$P_2$};
        
    \node (P6) [circle, fill = white, draw] at (9, 2.8) {$P_7$};
    \node (P7) [circle, fill = white, draw] at (9, 1.1) {$P_6$};
    
\end{tikzpicture}
\end{center}
The reader should convince themselves that both $QG_7 = (P_1, P_2, P_3, P_4, P_5, P_6, P_7)$ and $\sigma(QG_7) = (P_1, P_2, P_3, P_4, P_5, P_7, P_6)$ have the same corresponding graph, namely the graph $\Gamma_{QG_7}$ which we have drawn in the previous figure.

\begin{lemma}
Let $\Pi$ be a projective plane, and $G_k = (P_1, \dots, P_k)$ be a $k$-gon in $\Pi$. Then the group of symmetries of $G_k$ is precisely $D_k$.
\end{lemma}

\begin{proof}
In the following proof, all arithmetic is assumed to be taken modulo $k$. Let $D_k$ represent the subgroup of $S(G_k)$ described above. Suppose $\tau \in S_k$ such that for each $i$, $1 \leq i \leq k$, $\tau(i+1) = \tau(i)+1$ or $\tau(i+1) = \tau(i)- 1$. Then in fact, exactly one of the following must be true
\begin{itemize}
    \item  $\tau(i+1) = \tau(i) +1$ for all $1 \leq i \leq k$,
    \item  $\tau(i+1) = \tau(i) -1$ for all $1 \leq i \leq k$.
\end{itemize}
Indeed, if this was not the case, then there would exist a $j$ such that $\tau(j+1) = \tau(j) +1$ and $\tau(j+2) = \tau(j+1) - 1 = \tau(j) +1 - 1 = \tau(j)$. Clearly this is a contradiction, as $\tau(j+2) = \tau(j)$ implies $\tau$ is not a bijection. Note that by definition of $\sigma$ and $\rho$ in Lemma 1, any permutation satisfying either condition above must in fact be an element of $D_k$.

Suppose then that $\tau \in S_k \setminus D_k$, which implies that there exists an $i$, $1 \leq i \leq k$ for which $\tau(i + 1) \neq \tau(i) \pm 1$. For this $i$, let $\tau(i) = \ell$ and $\tau(i+1) = j$. Therefore in $\Gamma_{\tau(G_k)}$ we have $P_{\ell} \sim P_{\ell}P_j \sim P_j$. In $\Gamma_{G_k}$ we know that 
$$
P_{\ell - 1}P_{\ell} \sim P_{\ell} \sim P_{\ell}P_{\ell + 1} \hspace{1cm} \text{and} \hspace{1cm} P_{j-1}P_j \sim P_j \sim P_{j}P_{j+1}.
$$

If $\tau(G_k) \equiv G_k$, then we must have that $P_{\ell}P_j = P_{\ell - 1}P_{\ell}$ or $P_{\ell}P_j = P_{\ell}P_{\ell + 1}$ and $P_{\ell}P_j = P_{j-1}P_j$ or $P_{\ell}P_j = P_jP_{j+1}$. So we are left with four possibilities:
\begin{enumerate}
    \item $P_{\ell - 1}P_{\ell} = P_{j-1}P_j$.
    \item $P_{\ell}P_{\ell+1} = P_{j}P_{j+1}$.
    \item $P_{\ell - 1}P_{\ell} = P_{j}P_{j+1}$.
    \item $P_{\ell}P_{\ell + 1} = P_{j-1}P_j$.
\end{enumerate}

Each leads us to a contradiction, because all the lines of a $k$-gon are distinct. Cases 1 and 2 imply $\ell = j$, a contradiction. Case 3 implies $\ell - 1 = j$ and case 4 implies $\ell = j-1$, both are contradictions since we assumed $j \neq \ell \pm 1$. Thus, if $\tau \in S_k \setminus D_k$, then $G_k \not\equiv \tau(G_k)$, and so $S(G_k) = D_k$.
\end{proof}

\noindent \textbf{Remark:} Since the order of $S(G_k)$ is precisely $2k$ for any $k$-gon, then this demonstrates the assertion we made in the beginning of this section that 
$$
\frac{1}{2k}c_k(\Pi) = c_{2k}(\Gamma_{\Pi}).
$$




For positive integers $x$ and $k$, let $x_{(k)} = x(x-1)\cdots(x - k+1)$. Let $Q_k$ be the collection of all quasi $k$-gons in $\Pi$ of order $n$. Clearly, we have that $|Q_k| = N_{(k)}$ where $N = n^2 + n + 1$. Define $Q_{k, j} =  Q_{k, j}(\Pi) = \{ QG_k \in Q_k: |\mathcal{L}_{QG_k}| = j \}$. The the sets $Q_{k, j}$ form a partition of $Q_k$ and therefore 
\begin{equation}
|Q_k| = |Q_{k, k}| + |Q_{k, k-1}| + \cdots + |Q_{k, 1}|
\end{equation}
where in fact $|Q_{k, k}| = c_k(\Pi)$. Given this fact and the remark above, we may re-write $(1)$ in the form 
\begin{equation}
    c_{2k}(\Gamma_{\Pi}) = \frac{1}{2k}(|Q_k| - |Q_{k, k-1}| -  \cdots - |Q_{k, 1}|).
\end{equation}
Our end goal is to obtain equality of the first and second terms in the lower and upper bounds of $c_{2k}(\Gamma_{\Pi})$. We aim to do this by use of $(2)$, combined with bounds which we will obtain for $|Q_{k, j}|$ where $j < k$. We consider several cases, and obtain bounds in each case independently. We count:
\begin{enumerate}[(a)]
    \item The number of $QG_k$'s with $|\mathcal{L}_{QG_k}| = k-1$ and such that there exists an $m$ such that $P_mP_{m+1} = P_{m+1}P_{m+2}$. Let the set of all such quasi $k$-gons in $\Pi$ be denoted by $A_k(\Pi)$.
    \item The number of $QG_k$'s with $|\mathcal{L}_{QG_k}| = k-1$ and such that for all $m$, $P_mP_{m+1} \neq P_{m+1}P_{m+2}$. Let the set of all such quasi $k$-gons be denoted by $B_k(\Pi)$.
    \item The order of $Q_{k, j}$ for each $1 \leq j \leq k-2$.
\end{enumerate}


\begin{lemma}
Let $k \geq 4$ and $QG_k=(P_1, \dots, P_k)$ be a quasi $k$-gon with $|\mathcal{L}_{QG_k}| = k-1$. Suppose further that there exists an $m$ such that $P_mP_{m+1} = P_{m+1}P_{m+2}$. Then $|S(QG_k)| = 2k$.
\end{lemma}

\begin{proof}
Suppose $QG_k$ is as defined above. This implies that $P_m, P_{m+1}, P_{m+2}$ are collinear. Without loss of generality, we may assume $P_{m+1} = P_k$, since we may apply $\sigma = (12\dots k) \in S(QG_k)$ to $QG_k$, until we have moved $P_{m+1}$ into the position of $P_k$ and then relabel the points as $(P_1', \dots, P_k')$. We bring attention then to the fact that the lines  $P_iP_{i+1}$ are distinct for $1 \leq i \leq k-1$ and that the only lines in $QG_k$ equal to one another are $P_{k-1}P_k = P_kP_1$.


From here, we follow in the foot steps of the proof of Lemma 2. Recall that if $\tau \in S_k \setminus D_k$, then there exists an $i$ such that $\tau(i + 1) \neq \tau(i) \pm 1$.  For this $i$, let $\tau(i) = \ell$ and $\tau(i+1) = j$. This implies that in $\Gamma_{\tau(QG_k)}$ we have $P_{\ell} \sim P_{\ell}P_j \sim P_j$. On the other hand, in $\Gamma_{QG_k}$ we know that 
$$
P_{\ell - 1}P_{\ell} \sim P_{\ell} \sim P_{\ell}P_{\ell + 1} \hspace{1cm} \text{and} \hspace{1cm} P_{j-1}P_j \sim P_j \sim P_{j}P_{j+1}.
$$
If $\tau(QG_k) \equiv QG_k$, then we must have that $P_{\ell}P_j = P_{\ell - 1}P_{\ell}$ or $P_{\ell}P_j = P_{\ell}P_{\ell + 1}$ and $P_{\ell}P_j = P_{j-1}P_j$ or $P_{\ell}P_j = P_jP_{j+1}$.

\noindent So we are again left with the same four cases:
\begin{enumerate}
    \item $P_{\ell - 1}P_{\ell} = P_{j-1}P_j$
    \item $P_{\ell}P_{\ell+1} = P_{j}P_{j+1}$
    \item $P_{\ell - 1}P_{\ell} = P_{j}P_{j+1}$
    \item $P_{\ell}P_{\ell + 1} = P_{j-1}P_j$
\end{enumerate}

Case 1 implies that either $\ell = j$ or that $\ell = j -1 = k$. Case 2 implies that either $\ell = j$ or that $j = \ell - 1 = k$.
All of these options are contradictions, since $j \neq \ell$ and $j \neq \ell \pm 1$.

If we suppose case 3, then we obtain three possibilities. We may have $\ell - 1 = j$, a contradiction. We may have also consider $\ell = j = k$, again a contradiction. Therefore we are left with only one possibility, which is that $j = k-1$ and $\ell = 1$, in which case we obtain $P_{k}P_1 = P_{k-1}P_k$ which does not serve as a contradiction. Continuing on with this assumption, recall that $\tau(i) = \ell = 1$ and $\tau(i+1) = j = k-1$, so that in $\tau(QG_k)$, $P_1$ and $P_{k-1}$ are consecutive. Note then, for some $r$, $\tau(r) = k$, so $P_k$ is consecutive with $P_{\tau(r-1)}$ and $P_{\tau(r+1)}$ in $\tau(QG_k)$. Now, we must have that either $\tau(r-1) \neq 1, k-1$ or $\tau(r+1) \neq 1, k-1$. If not, then in $\tau(QG_k)$ we would observe either  $(\dots, P_1, P_k, P_{k-1}, \dots )$ or $(\dots, P_{k-1}, P_k, P_1, \dots)$, both cases which cannot happen since we have shown that in fact $P_{k-1} $ and $P_1$ must appear as consecutive elements in $\tau(QG_k)$. Without loss of generality, suppose $\tau(r-1) \neq 1, k-1$. Then the line $P_{\tau(r-1)}P_k$ appears in $\tau(QG_k)$. We know that in $QG_k$, $P_k$ has only one neighbor, namely $P_{k-1}P_k$, so then if $\tau(QG_k) \equiv QG_k$, then $P_{\tau(r-1)}P_k = P_{k-1}P_k$, which implies $P_{\tau(r-1)} \sim P_{k-1}P_k$ in $\Gamma_{\tau(QG_k)}$. This is a contradiction, since this edge does not appear in $\Gamma_{QG_k}$. Thus case 3 is also impossible. Case 4 follows in the same exact manner as case 3, except that the role of $j$ and $\ell$ is swapped.

Thus, each of the four cases leads us to a contradiction, implying that if $\tau \in S_k \setminus D_k$, then $\tau(QG_k) \not \equiv QG_k$. Then $S(QG_k) = D_k$.
\end{proof}

By Lemma 3, each $QG_k \in \Pi$ is equivalent to exactly $2k$ quasi $k$-gons in $A_k(\Pi)$, and so $\frac{1}{2k}|A_k(\Pi)|$ counts the number of equivalence classes of such quasi $k$-gons. 

We now comment on $\Gamma_{QG_k}$ for $QG_k \in A_k(\Pi)$. Let $QG_k = (P_1, \dots, P_k) \in A_k(\Pi)$, meaning for some $m$, $P_m, P_{m+1}, P_{m+2}$ are collinear. Then consider $G_{k-1} = (P_1, \dots, P_m, P_{m+2}, \dots, P_k)$ a $k-1$-gon. Note that $\mathcal{L}_{QG_k} = \mathcal{L}_{G_{k-1}}$ and that $\mathcal{P}_{QG_k} = \mathcal{P}_{G_{k-1}} \cup \{ P_{m+1} \}$.
Furthermore, if $P \sim \ell$ in $\Gamma_{G_{k-1}}$, then $P \sim  \ell$ in $\Gamma_{QG_k}$. We use the following example to illustrate the relationship between a $k$-gon and a quasi $(k+1)$-gon of the form described.

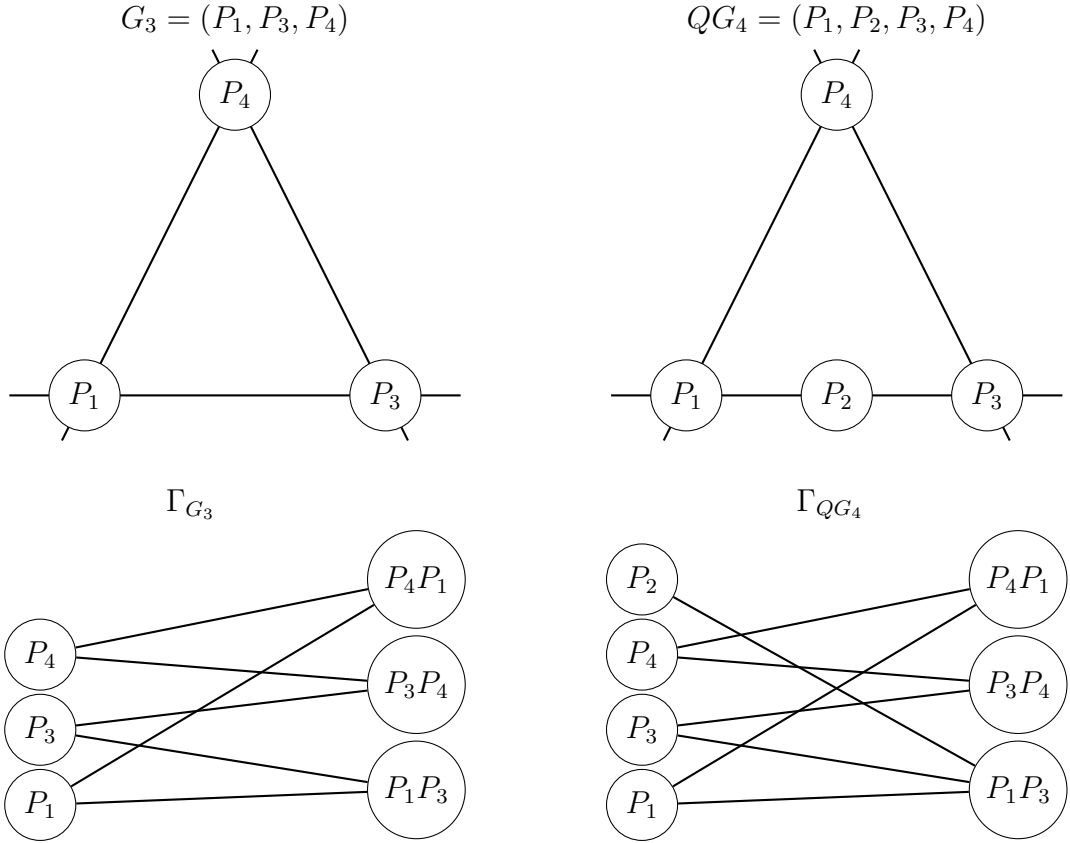
\begin{center}
\begin{tikzpicture}
    \draw [thick, draw = black]
        (0, -0.5) -- (6, -0.5)
        (0.7, -1.1) -- (3.3, 4.1)
        (5.3, -1.1) -- (2.7, 4.1)
       ;

    \node (P1) [circle, fill = white, draw] at (1, -0.5) {$P_1$};
    \node (P2) [circle, fill = white, draw] at (5, -0.5) {$P_3$};
    \node (P_3) [circle, fill = white, draw] at (3, 3.5) {$P_4$};
    
    \node (QG) [] at (3, 4.5) {$G_3 = (P_1, P_3, P_4)$};
    \node (G) [] at (11, 4.5) {$QG_4 = (P_1, P_2, P_3, P_4)$};
    
    \draw [thick, draw = black]
        (8, -0.5) -- (14, -0.5)
        (8.7, -1.1) -- (11.3, 4.1)
        (13.3, -1.1) -- (10.7, 4.1)
     ;

    \node (P1) [circle, fill = white, draw] at (9, -0.5) {$P_1$};
    \node (P4) [circle, fill = white, draw] at (11, -0.5) {$P_2$};
    \node (P2) [circle, fill = white, draw] at (13, -0.5) {$P_3$};
    \node (P_3) [circle, fill = white, draw] at (11, 3.5) {$P_4$};

\end{tikzpicture}
\end{center}

\begin{center}
\begin{tikzpicture}
     \node (PP1') [circle, fill = white, draw] at (1, -1) {$P_1$};
    \node (PP2') [circle, fill = white, draw] at (1, 0) {$P_3$};
    \node (PP3') [circle, fill = white, draw] at (1, 1) {$P_4$};

    \node (LL1) [circle, fill = white, draw] at (6, -0.8) {$P_1P_3$};
    \node (LL2) [circle, fill = white, draw] at (6, 0.6) {$P_3P_4$};
    \node (LL3) [circle, fill = white, draw] at (6, 2.0) {$P_4P_1$};

      \draw [ thick, draw=black] 
        (PP1') -- (LL1)
        (PP1') -- (LL3);
     \draw [ thick, draw=black] 
        (PP2') -- (LL1)
        (PP2') -- (LL2);
    \draw [ thick, draw=black] 
        (PP3') -- (LL2)
        (PP3') -- (LL3);
    
    \node (QG) [] at (3, 3) {$\Gamma_{G_3}$};
    \node (G) [] at (11.5, 3) {$\Gamma_{QG_4}$};

    \node (P1') [circle, fill = white, draw] at (9, -1) {$P_1$};
    \node (P2') [circle, fill = white, draw] at (9, 0) {$P_3$};
    \node (P3') [circle, fill = white, draw] at (9, 1) {$P_4$};
    \node (P4') [circle, fill = white, draw] at (9, 2) {$P_2$};

    \node (L1) [circle, fill = white, draw] at (14, -0.8) {$P_1P_3$};
    \node (L2) [circle, fill = white, draw] at (14, 0.6) {$P_3P_4$};
    \node (L3) [circle, fill = white, draw] at (14, 2.0) {$P_4P_1$};

      \draw [ thick, draw=black] 
        (P1') -- (L1)
        (P1') -- (L3);
     \draw [ thick, draw=black] 
        (P2') -- (L1)
        (P2') -- (L2);
    \draw [ thick, draw=black] 
        (P3') -- (L2)
        (P3') -- (L3);
    \draw [ thick, draw=black] 
        (P4') -- (L1);

\end{tikzpicture}
\end{center}

Since equivalence is directly tied to equality of graphs, we note that $\frac{1}{2k}|A_k(\Pi)|$ is counting exactly the number of distinct subgraphs of $\Gamma_{\Pi}$ that have the following form: A cycle of length $2k-2$ together with a distinct point vertex adjoined to exactly one line in the cycle. To be more rigorous, let $C$ be a cycle of length $2k-2$ in $\Gamma_{\Pi}$. We know that we may choose a $(k-1)$-gon, $G_{k-1} = (P_1, \dots, P_{k-1})$ such that $\Gamma_{G_{k-1}} = C$. If $P \not\in V(C)$ is a point vertex in $\Gamma_{\Pi}$ that is a neighbor of a line $\ell \in V(C)$, or equivalently, is a neighbor of a line $P_mP_{m+1} \in \Gamma_{G_{k-1}}$ for some $m$, then $\Gamma_{G_{k-1}}$ together with the vertex $P$ adjoined to $P_{m}P_{m+1}$, corresponds to the quasi $k$-gon $QG_k = (P_1, \dots, P_m, P, P_{m+1}, \dots, P_{k-1}) \in A_k(\Pi)$. So every graph of the form described, corresponds uniquely to an equivalence class of $QG_k$'s belonging to $A_k(\Pi)$. Denote this set of graphs by $A_{k}(\Gamma_{\Pi})$ so that $|A_k(\Gamma_{\Pi})| = \frac{1}{2k}|A_k(\Pi)|$.

\section{Bounds on $|A_k(\Gamma_{\Pi})|, |B_k(\Pi)|$ and $|Q_{k, j}|$}

\begin{lemma}
Let $\Pi$ be a projective plane of order $n$ and $n \geq k \geq 4$. Then
$$
(n-k+2)(k-1)c_{2k-2}(\Gamma_{\Pi})\leq |A_k(\Gamma_{\Pi})|\leq (n-1)(k-1)c_{2k-2}(\Gamma_{\Pi}).
$$
\end{lemma}

\begin{proof}
We may obtain bounds on $|A_k(\Gamma_{\Pi})|$ by counting the total number of cycles of length $2k-2$, and for each cycle adjoining a new point to one of the lines in the cycle. Let $C$ be a cycle of length $2k-2$. Each line in $V(C)$ has $n+1$ neighbors with at least $2$ of the neighbors in $V(C)$ and at most $k-1$ neighbors in $V(C)$. There are $k-1$ lines, and so we have that for each cycle we may create at least $(k-1)(n+1 - (k-1)) = (k-1)(n-k+2)$ distinct subgraphs belonging to $A_k(\Gamma_{\Pi})$, and at most $(k-1)(n-1)$. Clearly every graph in $A_k(\Gamma_{\Pi})$ can be obtained this way. Furthermore, since every graph in $A_k(\Gamma_{\Pi})$ corresponds to a unique cycle of length $2k-2$, then we have obtained each graph in $A_k(\Gamma_{\Pi})$ exactly once. Thus
$
(n-k+2)(k-1)c_{2k-2}(\Gamma_{\Pi})\leq |A_k(\Gamma_{\Pi})|\leq (n-1)(k-1)c_{2k-2}(\Gamma_{\Pi}).
$
\end{proof}

Let $QG_k = (P_1, P_2, \dots, P_k)$. We may associate with $QG_k$ a sequence of $k$ lines of the form $QG_k = (P_1P_2, P_2P_3, \dots, P_kP_1) = (\ell_1, \ell_2, \dots, \ell_k)$. Note the lines may not necessarily be distinct, and in the case that lines are repeated it is possible for distinct quasi $k$-gons to have the same line sequence. For the next two lemmas we will use this correspondence to obtain bounds on the remaining cases. We remind the reader that $Q_{k, j} = \{ QG_k \in Q_k : | \mathcal{L}_{QG_k} | = j \}$ and $B_k(\Pi) = Q_{k, k-1} \setminus A_k(\Pi)$.

\begin{lemma}
Let $\Pi$ be a projective plane of order $n$ and $k \geq 4$. Then
$$
|B_k(\Pi)| \leq (k-1)(k-2)N_{(k-1)}.
$$
\end{lemma}

\begin{proof}
Let $QG_k = (P_1, P_2, \dots, P_k) \in B_k(\Pi)$ and let $(\ell_1, \ell_2, \dots, \ell_k)$ be its corresponding sequence of lines. By definition of $B_k(\Pi)$, we have that $|\mathcal{L}_{QG_k}| = k-1$ and for all $i$, $\ell_i \neq \ell_{i+1}$. This implies that we may describe each point $P \in \mathcal{P}_{QG_k}$ as an intersection of consecutive lines in the sequence $(\ell_1, \ell_2, \dots, \ell_k)$. More specifically, we have $P_{i} = \ell_{i-1} \cap \ell_{i}$ for all $i$. Therefore, distinct quasi $k$-gons must in fact have distinct line sequences.

We may then use the sequence of lines $(\ell_1, \ell_2, \dots, \ell_k)$ as an alternative description of $QG_k$. We now find an upper bound on the number of line sequences that correspond to quasi $k$-gons in $B_k(\Pi)$. In order to do this, we first consider all sequences of $(k-1)$ distinct lines. There are $N_{(k-1)}$ many such sequences. Then for each such sequence, we create sequences of $k$ lines in the following way. Pick any line in the sequence of $(k-1)$ lines, which can be done in $k-1$ ways. Then insert it into the same sequence such that it will not appear next to itself, which can be done in at most $k-2$ ways. Therefore, the number of possible such sequences in total is at most $(k-1)(k-2)N_{(k-1)}$. Clearly, any line sequence corresponding to $QG_k \in B_k(\Pi)$ can be obtained this way, so that $(k-1)(k-2)N_{(k-1)}$ is an upper bound on $|B_k(\Pi)|$.
\end{proof}

Let $QG_k = (P_1, \dots, P_k)$  with the corresponding sequence of lines $(\ell_1, \ell_2, \dots, \ell_k)$ with exactly $j$ lines being distinct. We will call a subsequence $(P_i, P_{i+1}, \dots, P_{i+t})$ of consecutive points of $QG_k$ a \textit{maximal subsequence} if all its points lie on some line $\ell$, but $P_{i-1}$ and $P_{i+t+1}$ do not lie on $\ell$. Observe that if $QG_k$ has a maximal subsequence of length $t+1$, then the corresponding line sequence of $QG_k$ has a subsequence of $t$ consecutive equal lines. We will also refer to this subsequence of lines as maximal. 

\begin{lemma}
Let $k$ and $j$ be fixed positive integers, with $k \geq 4$ and $k-2 \geq j \geq 2$. Let $\Pi$ be a projective plane of order $n$, with $n \geq k$. Then
$$
|Q_{k, j}| \leq j^{k-j}k_{(j)}\binom{N}{j}(n-1)^{k-j} = O(n^{j+k}), \hspace{0.5cm} n \rightarrow \infty
$$
and
$$
|Q_{k, 1}| = N(n+1)_{(k)}.
$$
\end{lemma}

\begin{proof}
We follow similar line of reasoning as in the proof of Lemma 5 above. To obtain an upper bound for $|Q_{k, j}|$, we first obtain an upper bound on the number of distinct sequences of $k$ lines made out of $j$ distinct lines. We then proceed by placing an upper bound on the number of distinct $QG_k \in Q_{k, j}$ that can have the same corresponding sequence of lines.

We build the sequences in the following way. Begin with an empty sequence of $k$ available positions. Choose any $j$ disinct lines in $\Pi$. There are $\binom{N}{j}$ ways to do this. Then out of the $k$ available posiitons in the empty sequence, choose $j$ of them, and place the $j$ chosen lines into these positions in any order. There are $k_{(j)}$ ways to do this. For every remaining empty position, place one of the chosen $j$ lines to fill it. There are $k-j$ empty spots, and $j$ possible choices for each, giving us $j^{k-j}$ ways to fill all the remaining positions. So there are at most $j^{k-j}\binom{N}{j}k_{(j)}$ different sequences of lines with exactly $j$ distinct lines appearing in the sequence.



Let $QG_k = (P_1, P_2, \dots, P_k) \in Q_{k, j}$ and suppose that  $(P_i, P_{i+1}, \dots, P_{i+t})$ is a maximal subsequence of $QG_k$. As $j \geq 2$, we know that $t+1 < k$, meaning that this maximal subsequence is not equal to $QG_k$. Let $P_iP_{i+1} = \ell$ and consider $QG_k' = (P_1, \dots, P_i, P_{i+1}', \dots, P_{i+t-1}', P_{i+t}, \dots, P_k)$ where $(P_{i+1}', \dots, P_{i+t-1}')$ are any sequence of distinct points from $\ell \setminus \{ P_{i}, P_{i+t}\}$. Observe that $QG_k$ and $QG_k'$ have the same corresponding line sequence. There are $(n-1)_{t-1}$ different $QG_k'$ that can be obtained in this way. In summary, for every maximal subsequence of length $t+1$ in $QG_k$, there are $(n-1)_{t-1}$ quasi $k$-gons that have the same line sequence as $QG_k$ and only differ from $QG_k$ by the interior points in the maximal subsequence. 

Recall that in the line sequence of $QG_k$, each maximal subsequence of $QG_k$ of length $t+1$ corresponds to a maximal subsequence of length $t$ in the line sequence of $QG_k$. Any two maximal subsequences contained in the line sequence are necessarily disjoint. Furthermore, we may uniquely partition the line sequence by maximal subsequences. 

Let $L = (\ell_1, \dots, \ell_k)$ be a sequence of $k$ lines that corresponds to some $QG_k \in Q_{k, j}$. Partition $L$ into its maximal subsequences. If the number of parts is $r$, then it is clear that $r \geq j \geq 2$. Let $t_s$ be the length of each maximal subsequence, $1 \leq s \leq r$, then $t_1 + t_2 + \dots +t_r = k$. For each maximal subsequence of length $t_s$, there are at most $(n-1)_{(t_s - 1)}$ distinct quasi $k$-gons that have the same line sequence as $QG_k$ and differ from $QG_k$ only by the interior points of the corresponding maximal subsequence of $t_s +1$ points of $QG_k$. Then the total number of $QG_k$'s that have $L$ as their line sequence is given by the product
$$
\prod_{s = 1}^r (n-1)_{(t_s - 1)} \leq (n-1)^{(\sum_{s=1}^r (t_s - 1))} = (n-1)^{k - r} \leq (n-1)^{k-j}.
$$

When $j = 1$, this means that the sequence has only one distinct line. The number of such quasi $k$-gons is easy to count, as there are exactly $N$ lines in $\Pi$ and there are $(n+1)_{(k)}$ sequences of points that we can form from each line.
\end{proof}

\begin{theorem}
Let $k$ be a fixed positive integer with $k \geq 4$. Let $\Pi$ be a finite projective plane of order $n$, with $n \geq k$, and $\Gamma_{\Pi}$ be its Levi graph. Then
$$
c_{2k}(\Gamma_{\Pi}) > \frac{1}{2k}N_{(k)} - \frac{1}{2}(n-1)N_{(k-1)} -\frac{(k-1)(k-2)}{2k}N_{(k-1)} - a_kn^{2k-2}.
$$
where $a_k$ is a constant dependent only on $k$.
\end{theorem}

\begin{proof}
Recall that 
$$
c_{2k}(\Gamma_{\Pi}) = \frac{1}{2k}(|Q_k| - |Q_{k, k-1}| -  \cdots - |Q_{k, 1}|).
$$
Clearly $c_{2k-2}(\Pi) \leq |Q_{k-1}| = \frac{1}{2(k-1)}N_{(k-1)}$, which together with Lemma 4 implies that 
$$
\frac{1}{2k}|A_k(\Pi)|  = |A_k(\Gamma_{\Pi})| \leq (n-1)(k-1)c_{2k-2}(\Gamma_{\Pi}) \leq \frac{1}{2}(n-1)N_{(k-1)}.
$$
Furthermore, as $A_k(\Pi) \dot\cup B_k(\Pi) = Q_{k, k-1}$, by applying the inequality above to $|A_k(\Pi)|$ and Lemma 5 to $|B_k(\Pi)|$, we have
$$
\frac{1}{2k}|Q_{k, k-1}| = \frac{1}{2k}|A_k(\Pi)| + \frac{1}{2k}|B_k(\Pi)| \leq \frac{1}{2}(n-1)N_{(k-1)} + \frac{(k-1)(k-2)}{2k}N_{(k-1)}.
$$
Finally, Lemma 6 implies that 
$$
\sum_{j = 1}^{k-2}|Q_{k, j}| \leq a_kn^{2k-2}
$$
for some constant $a_k$ dependent only on $k$.
Combining these results yields the claimed lower bound for $c_{2k}(\Gamma_{\Pi})$.
\end{proof}

Now we present an upper bound for $c_{2k}(\Gamma_{\Pi})$.

\begin{theorem}
Let $\Pi$ be a finite projective plane of order $n$ and $\Gamma_{\Pi}$ be its Levi graph. If $k$ is a fixed integer with $n \geq k \geq 4$, then
$$
c_{2k}(\Gamma_{\Pi}) \leq \frac{1}{2k}N_{(k)} - (n-k+2)(k-1)c_{2k-2}(\Gamma_{\Pi})
$$
\end{theorem}

\begin{proof}
Observe that
$$
c_{2k}(\Gamma_{\Pi}) = \frac{1}{2k}(|Q_k| - |Q_{k, k-1}| - \cdots - |Q_{k, 1}|) \leq \frac{1}{2k}|Q_k| - \frac{1}{2k}|Q_{k, k-1}|.
$$
As $|Q_{k, k-1}| = |A_k(\Pi)| + |B_{k}(\Pi)|$, then by Lemma 4,
$$
\frac{1}{2k}|Q_{k, k-1}| \geq \frac{1}{2k}|A_k(\Pi)| = |A_{k}(\Gamma_{\Pi})| \geq (n- k + 2)(k-1)c_{2k-2}(\Gamma_{\Pi}).
$$
Recalling that $|Q_k| = N_{(k)}$, we obtain
$$
c_{2k}(\Gamma_{\Pi}) \leq \frac{1}{2k}N_{(k)} - (n-k+2)(k-1)c_{2k-2}(\Gamma_{\Pi}).
$$
\end{proof}

\section{Proofs of Theorems 1 and 2}

\noindent We now have all the tools we need to prove Theorem 1. For convenience we restate the theorem here.

\bigskip

\noindent\textbf{Theorem 1.}
{\em 
Let $\Pi$ be a projective plane of order $n$ and $\Gamma_{\Pi}$ be its Levi graph. Then for fixed $k \geq 4$,}
$$
c_{2k}(\Gamma_{\Pi}) = \frac{1}{2k}n^{2k} +O(n^{2k-2}), \hspace{1cm} n \rightarrow \infty
$$

\begin{proof}
All we need to demonstrate, is that the coefficient of $n^{2k-1}$ is 0. Theorem 3 states that 
$$
c_{2k}(\Gamma_{\Pi}) > \frac{1}{2k}N_{(k)} - \frac{1}{2}(n-1)N_{(k-1)} -\frac{(k-1)(k-2)}{2k}N_{(k-1)} - a_kn^{2k-2}.
$$
Recalling that $N = n^2 + n +1$, observe that the only terms that contribute to the coefficient of $n^{2k-1}$ come from 
$$
\frac{1}{2k}N_{(k)} - \frac{1}{2}(n-1)N_{(k-1)}.
$$
It is easy to see that both $\frac{1}{2k}N_{(k)}$ and $\frac{1}{2}(n-1)N_{(k-1)}$ have $1/2$ as a coefficient of $n^{2k-1}$. This implies that the our lower bound has $0$ as the coefficient of $n^{2k-1}$.

Now, Theorem 3 gave us that 
$$
c_{2k}(\Gamma_{\Pi}) \leq \frac{1}{2k}N_{(k)} - (n-k+2)(k-1)c_{2k-2}(\Gamma_{\Pi}).
$$
By applying Theorem 4 to $c_{2k-2}(\Gamma_{\Pi})$ we obtain an upper bound on $c_{2k}(\Gamma_{\Pi})$ where the coefficient of $n^{2k-1}$ is again 0.
\end{proof}

\begin{cor}
Let $k \geq 4$, then
$$
\text{ex}(v, C_{2k},   \mathcal{C}_{\text{odd}}  \cup \{ C_4 \}) \geq \left(\frac{1}{2^{k+1}k}-o(1)\right)v^k, \hspace{1cm} v \rightarrow \infty.
$$
\end{cor}

\begin{proof}
Theorem 1 implies that the stated result is in fact true when $v = 2(n^2 + n +1)$ where $n$ is a prime power with $n \geq k$. In \cite{Bake}, Baker, Harman, and Pintz show that for all sufficiently large $v$, there exists a prime in the interval $[v - v^{0.525}, v]$. A standard argument using this fact yields the corollary for all sufficiently large $v$.
\end{proof}

\begin{theorem}
Let $k \geq 3$, then
$$
\text{ex}(v, C_{2k}, \mathcal{C}_{\text{odd}} \cup \{ C_4 \}) \leq \frac{1}{2k}\left(\frac{v}{2}\right)_{(k)}.
$$
\end{theorem}

\begin{proof}
Let $G$ be a bipartite graph containing no $C_4$ as a subgraph. As $G$ is bipartite, we may partition $V(G)$ into two independent sets. Let $A$ represent the smaller of the two parts, so that $|A| \leq v/2$. Let $G^2$ be the graph obtained from $G$ in the following manner: $V(G^2) = A$ and if $x, y$ are distinct vertices in $A$, then $x \sim y$ in $G^2$ when there exists $z \in V(G)$ such that $x \sim z \sim y$ in $G$. By definition, $G^2$ has no loops and since $G$ has no $C_4$, then $G^2$ has no multiple edges, implying that $G^2$ is simple.

Suppose that $C_{2k}$ is some cycle of length $2k$ in $G$. As $G$ is bipartite, exactly $k$ vertices of $C_{2k}$ are in $G$, and as a result, every cycle of length $2k$ has a unique corresponding cycle of length $k$ in $G^2$. Therefore, the number of $k$ cycles in $G^2$, is an upper bound on the number of $2k$-cycles in $G$. As $G^2$ is simple, the number of $k$ cycles in $G^2$ is no more than in a complete graph of order $|G^2|$. Thus the number of cycles of length $2k$ in $G$ is no more than
$$
\frac{1}{2k}|A|_{(k)} \leq \frac{1}{2k}\left(\frac{v}{2}\right)_{(k)}.
$$
\end{proof}
\noindent Finally, combining our results above, Theorem 2 falls out as a simple corollary. 

\bigskip

\noindent\textbf{Theorem 2.}
{\em Let $k \geq 4$ and $n \geq k$ be a prime power. If $v = 2(n^2 + n + 1)$, then} 
$$
\text{ex}(v, C_{2k}, \mathcal{C}_{\text{odd}} \cup \{ C_4 \}) = \left(\frac{1}{2^{k+1}k}-o(1)\right)v^k, \hspace{1cm} v \rightarrow \infty.
$$

\section*{Concluding Remarks}
At this time, it is not known whether $c_{2k}(\Gamma_{\Pi})$ is a polynomial in $n$ for all $k$. For $k  = 3, 4, \dots, 10$, the exact value of $c_{2k}(\Gamma_{\Pi})$ was determined in  $\cite{LazMelV2}, \cite{Vor}$, and it was shown that the function $c_{2k}(\Gamma_{\Pi})$ is a polynomial of degree $2k$. Let $k = 3, 4, \dots, 10$ and $\Pi$ be a projective plane of order $n$, and 
$$
c_{2k}(\Gamma_{\Pi}) = a_{2k}n^{2k} + a_{2k - 1}n^{2k-1} + \cdots + a_0.
$$
The following table records the coefficient data for the first, second, third and fourth coefficients, namely the coefficients of $n^{2k}, n^{2k-1}, n^{2k-2}$ and $n^{2k-3}$ respectively.

\newpage

\begin{table}[h!]
    \centering
    \begin{tabular}{c|c|c|c|c}
           & $a_{2k}$ & $a_{2k-1}$ & $a_{2k-2}$  & $a_{2k-3}$  \\ \hline
    $k=3$ & 1/6 & 1/3 & 1/3 & 1/6 \\ \hline
    $k=4$ & 1/8 & 0 & -1/8 & -1/8 \\ \hline
    $k=5$ & 1/10 & 0 & 0 & -1/10 \\ \hline
    $k=6$ & 1/12 & 0 & -1/2 & 0 \\ \hline
    $k=7$ & 1/14 & 0 & -1 & 3/2 \\ \hline
    $k=8$ & 1/16 & 0 & -3/2 & 3 \\ \hline
    $k=9$ & 1/18 & 0 & -2 & 9/2 \\ \hline
    $k=10$ & 1/20 & 0 & -5/2 & 6 \\  \hline
    \end{tabular}
    \caption{The 1st, 2nd, 3rd, and 4th coefficients in the formula for $c_{2k}(\Gamma_{\Pi})$ up to $k = 10$.}
    \label{tab:my_label}
\end{table}

There is a pattern that can be observed for each coefficient in the table and for sufficiently large $k$. For $a_{2k}$, the pattern begins at $k= 3$. For $a_{2k-1}$, the pattern begins at $k = 4$. For $a_{2k-2}$, the patter begins at $k = 5$, in which case we see a decrease by 1/2 every successive $k$ after $k = 5$. For $a_{2k-3}$, we notice an increment of $3/2$ every successive $k$ after $k = 6$. In light of the information obtained from the table, we state the following conjecture.
\begin{conjecture}
Let $k$ be a fixed positive integer with $k \geq 6$. Let $\Pi$ be a projective plane of order $n$ and $\Gamma_{\Pi}$ it's Levi graph, then
$$
c_{2k}(\Gamma_{\Pi}) = \frac{1}{2k}n^{2k} -\frac{1}{2}(k-5)n^{2k-2} + \frac{3}{2}(k-6)n^{2k-3} + O(n^{2k-4}), \hspace{0.5cm} n \rightarrow \infty.
$$
\end{conjecture}
We would like to conclude with the question posed at the beginning of this section:
\begin{question}
Let $k$ be a fixed positive integer with $k \geq 3$. Let $\Pi$ be a projective plane of order $n$ and $\Gamma_{\Pi}$ it's Levi graph, is it the case that $c_{2k}(\Gamma_{\Pi})$ is a polynomial in $n$?
\end{question}

\section*{Acknowledgments}

The author would like to thank Dr. Felix Lazebnik and Dr. Eric Moorehouse for posing the problem and for their valuable comments that improved both the results and the exposition of this paper.


\begin{thebibliography}{}

\bibitem{Ace}
E.\ Aceves, D.\ Heywood, A.\ Klahr and O.\ Vega,
{\em Cycles in projective spaces},
Journal of Geometry, 105, (2014), 111--117.

\bibitem{Aloshil}
N.\ Alon and C.\ Shikhelman,
{\em Many $T$ copies in $H$-free graphs},
Electronic Notes in Discrete Mathematics, 49, (2015), 683--689.

\bibitem{Bake}
R. C. Baker, G. Harman, and J. Pintz,
{\em The Difference Between Consecutive Primes, II},
Proceedings of the London Mathematical Society, 83, (2001), 532--562.

\bibitem{NBig}
N. Biggs, {\em Algebraic Graph Theory},
Cambridge University Press, (1993).

\bibitem{BG}
B.\ Bollob\'{a}s,
{\em Modern Graph Theory},
Springer Science and Business Media, (1998).

\bibitem{Bol2}
B.\ Bollob\'{a}s,
{\em Extremal Graph Theory},
Academic Press London, (1978).




\bibitem{Casse}
R.\ Casse, 
{\em Projective Geometry: An Introduction},
Oxford University Press, (2006).

\bibitem{Wint}
S.\ De Winter, F.\ Lazebnik and J. Verstra\"{e}te,
{\em An Extremal Characterization of Projective Planes},
Electronic Journal of Combinatorics, 15(1), (2008), 1--13.

\bibitem{Erd1}
P.\ Erd\"{o}s,
{\em On some problems in graph theory, combinatorial analysis and combinatorial number theory},
Graph Theory and Combinatorics, (1983), 1--17.

\bibitem{FioL}
G.\ Fiorini and F.\ Lazebnik,
{\em On a bound for the number of $C_8's$ in 4-cycle free bipartite graph}, 
Congressus Numerantium , 99, (1994), 191--197.

\bibitem{Fur}
Z.\ F\"{u}redi,
{\em Tur\'{a}n type problems}, 
Survey's in Combinatorics, Cambridge University Press, (1991).

\bibitem{FuS}
Z.\ F\"{u}redi and M. Simonovits
{\em The history of degenerate(bipartite) extremal graph problems},
Erdo\"{o}s Cenetnnial, Bolyai Society Mathematical Studies, 25, (2013), pp169--264.

\bibitem{GerPal}
D.\ Gerbner and C.\ Palmer,
{\em Counting copies of a fixed subgraph in $F$-free graphs},
European Journal of Combinatorics, 82, (2019).

\bibitem{GerG}
D.\ Gerbner, E.\ Gy\"{o}ri, A.\ Methuku and M.\ Vizer,
{\em Generalized Tur\'{a}n Problems for Even Cycles},
Journal of Combinatorial Theory, Series B, 145, (2020), 169--213

\bibitem{Glyn}
D.\ G.\ Glynn,
{\em Rings of Geometries II},
Journal of Combinatorial Theory, Series A, 49(1), (1988), 26--66.

\bibitem{Grez}
A.\ Grzesik,
{\em On the maximum number of five-cycles in a triangle free graph},
Journal of Combinatorial Theory, Series B, 102(5), (2012), 1061--1066.

\bibitem{Gyor}
E.\ Gy\"{o}ri,
{\em On the number of $C_5$'s in a triangle-free graph},
Combinatorica, 9, (1989), 101--102.

\bibitem{Hatam}
H.\ Hatami, J.\ Hladk\'{y}, D.\ Kr\'{a}l, S.\ Norine and A.\ Razborov,
{\em On the number of pentagons in triangle-free graphs},
Journal of Combinatorial Theory Series A, 120(3),  (2013), 722--732.

\bibitem{Ver}
J.\ Verstra\"{e}te, 
{\em A survey of Tur\'{a}n problems for expansions},
In: Recent Trends in Combinatorics, (2016), 83--116. 


\bibitem{Kap}
N.\ Kaplan, S.\ Kimport, R.\ Lawrence, L.\ Peilen and M. Weinreich,
{\em Counting arcs in projective planes via Glynn's algorithm },
Journal of Geometry, 108, (2017), 1013--1029.


\bibitem{LazMelV1}
F.\ Lazebnik, K.\ Mellinger and O.\ Vega,
{\em Embedding Cycles in Finite Planes},
Electronic Journal of Combinatorics, 20(3), 2013.

\bibitem{LazMelV2}
F.\ Lazebnik, K.\ Mellinger and O.\ Vega,
{\em On the Number of $k$-Gons in Finite Projective Planes},
Note di Matematica, 29, (2009), 135--152.

\bibitem{FelSun}
F.\ Lazebnik, S.\ Sun and Y.\ Wang,
{\em Some Families of Graphs, Hypergraphs and Digraphs Defined by Systems of Equations: A Survey.},
Lecture Notes of Seminario Interdisciplinare di Matematica, 14, (2017), 105–-142.


\bibitem{MVer}
D.\ Mubayi and J.\ Verstra\"{e}te, 
{\em A survey of Tur\'{a}n problems for expansions},
In: Recent Trends in Combinatorics, (2016), 117--143. 


\bibitem{SolW}
J.\ Solymosi and C.\ Wong,
{\em Cycles in graphs of fixed girth with large size},
European Journal of Combinatorics, 62, (2017), 124--131.



\bibitem{Vor}
A.\ N.\ Voropaev,
{\em Counting $k$-gons in finite projective planes},
Siberian Electronic Mathematical Reports, 10, (2013), 241--270.




 

\end{thebibliography}
\end{document}